\documentclass[12pt,reqno]{amsart}
\usepackage[T1]{fontenc}
\usepackage{lmodern}
\usepackage[margin=1in]{geometry}
\usepackage{microtype}
\usepackage{mathtools,amssymb}
\usepackage[plain]{algorithm}
\usepackage{caption}
\usepackage{algpseudocode}
\algrenewcommand{\algorithmicrequire}{\textbf{Input:}}
\algrenewcommand{\algorithmicensure}{\textbf{Output:}}
\usepackage{float}
\newtheorem*{theorem}{Theorem}
\newtheorem{lemma}{Lemma}
\newcommand{\NN}{\mathbb N}
\newcommand{\ZZ}{\mathbb Z}
\newcommand{\bits}{\mathrel{\triangleleft}}
\newcommand{\C}[1]{\mathcal C(#1)}

\usepackage{xcolor}
\usepackage[bookmarks=false]{hyperref} 
\definecolor{bluey}{rgb}{0,0,0.6}
\hypersetup{
  pdftitle={A 4AP-free permutation of the positive integers},
  pdfauthor={Boon Suan Ho},
  colorlinks,
  linkcolor=bluey,
  citecolor=bluey,
  urlcolor=bluey
}

\makeatletter
\def\thmhead@plain#1#2#3{%
  \thmname{#1}\thmnumber{\@ifnotempty{#1}{ }\@upn{#2}}%
  \thmnote{ \the\thm@notefont(#3)}}
\let\thmhead\thmhead@plain
\makeatother

\title{A 4AP-free permutation of the positive integers}
\author{Boon Suan Ho}
\email{hbs@u.nus.edu}
\date{}

\begin{document}

\begin{abstract}
We construct a permutation of the positive integers containing no four-term
arithmetic progression as a subsequence.
\end{abstract}

\maketitle

Davis, Entringer, Graham, and Simmons~\cite{DEGS} showed that every permutation
$a_1a_2\dots$ of the positive integers $\NN$ contains a monotone three-term
arithmetic progression (3AP); that is, indices $i<j<k$ such that
$(a_i,a_j,a_k)=(a_i,a_i+r,a_i+2r)$ for some integer $r\ne0$. They also
constructed a permutation of $\NN$ avoiding 5APs, and asked whether there exists
a 4AP-free permutation of $\NN$. Towards answering this question,
LeSaulnier--Vijay~\cite{LV} and Adenwalla~\cite{Adenwalla} constructed
permutations of $\NN$ avoiding 4APs with restricted common differences.
Adenwalla~\cite[Theorem~6]{Adenwalla} also proved that the supremal lower
density of subsets of $\NN$ admitting 4AP-free permutations is one. 

We resolve the question by constructing a 4AP-free permutation of the
positive integers.

\begin{theorem}
There exists a permutation of the positive integers containing no subsequence of
the form
$$
 a,\ a+r,\ a+2r,\ a+3r,
$$
where $r$ is a nonzero integer.
\end{theorem}
We construct the permutation on $\NN_0=\{0,1,2,\ldots\}$; we can then add one to
the resulting permutation to obtain the desired permutation of $\NN$.

We will call a strict linear order $\prec$ on $\NN_0$ \emph{4AP-free}
if there are no $a\in\NN_0$ and $r\in\ZZ\setminus\{0\}$ with
$a,a+r,a+2r,a+3r\in\NN_0$ such that $a\prec a+r\prec a+2r\prec a+3r$.

Our proof relies on the following linear order $\bits$ of $\NN_0$. Given
distinct nonnegative integers $m=(\dots m_2m_1m_0)_2$ and $n=(\dots
n_2n_1n_0)_2$ expanded in binary, if $k$ is the smallest index for which $m_k\ne
n_k$, we set $m\bits n$ iff $m_k>n_k$. For example, $\bits$ restricted to
$\{0,\ldots,7\}$ is
$$
 7\bits3\bits5\bits1\bits6\bits2\bits4\bits0,
$$
the reverse colexicographic order on $3$-bit strings.

Note that all odd numbers precede all even numbers under $\bits$. Moreover,
after restricting $\bits$ to either parity and rescaling by $2x+1\mapsto x$ or
$2x\mapsto x$, we recover the same order $\bits$; in other words, $2x\bits 2y$ iff $x\bits y$ iff $2x+1\bits 2y+1$. This self-similarity will
allow us to treat odd and even parts recursively later in our proof. Finally,
note that $0$ is the greatest element of $\bits$.

The order $\bits$ is the dual of the restriction to $\NN_0$ of the binary
chaotic order of Ardal--Brown--Jungi\'c \cite[Definition~2.2 and \S5,
Remark~2]{ABJ}. Its 3AP-freeness follows from~\cite[Theorem~2.2]{ABJ}, and the
comparison identity~\eqref{eq:pairs} given below is a special case of
\cite[Lemma~2.5]{HS}; both of these properties are recorded in
\cite[Lemma~2.2]{Geneson}. We now give elementary proofs for completeness.

If an arithmetic progression has nonzero difference $r=2^kq$ with $q$ odd, its
terms agree at all bits below the $k$th, and alternate at the $k$th bit. Thus
three consecutive terms of an arithmetic progression have $k$th bits $0,1,0$ or
$1,0,1$. Since all lower bits agree, $k$ is the smallest index at which
consecutive terms differ. Hence the middle term either precedes both endpoints
in $\bits$, or follows both endpoints. Therefore neither $\bits$ nor its dual
contains a 3AP.

Similarly, if $a,b,c,d$ are four consecutive terms of an arithmetic progression,
their $k$th bits are $0,1,0,1$ or $1,0,1,0$. Thus
\begin{equation}\label{eq:pairs}
 a\bits b\quad\Longleftrightarrow\quad c\bits d.
\end{equation}

For a finite word $P$ of distinct nonnegative integers, let $\C P$ be the order
obtained by putting $P$ first and then ordering the remaining integers
$\NN_0\setminus P$ by $\bits$. We call $\C P$ the \emph{completion} of $P$, and we say that $P$ is \emph{safe} if $\C P$ is 4AP-free.
Starting from the empty word, we will use the extension lemma below to construct
nested safe words containing successively prescribed finite sets. Taking the
prescribed sets to be $\{0\},\{1\},\{2\},\ldots$ will then yield the desired
permutation.

The following lemma uses an argument analogous to that in~\cite[proof of
Lemma~5.6]{Geneson}.

\begin{lemma}\label{lem:base}
Every finite set, listed in reverse $\bits$-order, is a safe word.
\end{lemma}

\begin{proof}
Let $S$ be the finite set and let $P$ denote its elements in reverse
$\bits$-order. Recall that $\C P$ consists of $P$, followed by all integers
outside $S$ in $\bits$-order.

Suppose for contradiction that a 4AP $a,b,c,d$ occurs in this order in $\C P$.
Since every element of $P$ precedes every element outside $P$, the terms of the
progression which lie in $P$ form an initial segment of $a,b,c,d$.

If at most one of the four terms lies in $P$, then $b,c,d$ all lie in the
$\bits$-ordered part, contradicting the fact that $\bits$ is 3AP-free. If at
least three lie in $P$, then $a,b,c$ lie in the reverse $\bits$-ordered part,
which is also 3AP-free.

Thus necessarily $a,b\in P$ and $c,d\notin P$. Since $P$ is ordered in reverse
$\bits$-order, $b\bits a$, whereas the remaining integers are ordered by
$\bits$, so $c\bits d$. But this contradicts~\eqref{eq:pairs}.
\end{proof}

The extension lemma below is the main step of our construction. Its proof uses
the self-similarity of $\bits$: we treat the even and odd parts recursively, at
each stage determining which new even numbers are needed, and then placing
sufficiently many odd numbers before them to prevent 4APs from forming.

\begin{lemma}[Extension]\label{lem:extend}
Every safe finite word $P$ is an initial segment of a safe finite word
containing any prescribed finite set $T\subset\NN_0$.
\end{lemma}

\begin{proof}
We induct on the largest entry $m$ of $P$, taking $m=0$ if $P$ is empty. For
each value of~$m$, we require the statement to hold for every finite target set
$T$.

If $m=0$, then $P$ is either empty or consists only of $0$. List the elements of
$P\cup T$ without repetition in reverse $\bits$-order. Since $0$ is the
greatest element of $\bits$, this word begins with $P$, and it is safe by
Lemma~\ref{lem:base}.

Now suppose $m>0$. Split $P$ according to parity. Let $P_0$ be the word obtained
from the even entries of $P$ by dividing them by $2$, and let $P_1$ be obtained
from the odd entries by applying $2x+1\mapsto x$. Both $P_0$ and $P_1$ are safe.
Indeed, restricting $\C P$ to the even integers leaves first the even entries
of $P$, followed by the unused even integers in $\bits$-order. After applying
$2x\mapsto x$, the first block becomes $P_0$, while the second becomes
$\NN_0\setminus P_0$ in $\bits$-order, by the self-similarity of $\bits$.
Thus the resulting order is exactly $\C{P_0}$.
Similarly, after restricting $\C P$ to the odd integers and rescaling by
$2x+1\mapsto x$, we obtain exactly $\C{P_1}$. Since a restriction of a 4AP-free
order is 4AP-free, both $P_0$ and $P_1$ are safe. Moreover, every entry of
either word is at most $\lfloor m/2\rfloor<m$, so the induction hypothesis
applies to both.

\newpage
\begin{algorithm}[H]
\small
\hrule height .8pt
\kern 3pt
\noindent$\textsc{Extend}(P,T)$\par
\kern 3pt
\hrule
\begin{algorithmic}[1]
\Require A safe finite word $P$ and a finite target set $T\subset\NN_0$.
\Ensure A safe word $Q$ having $P$ as a prefix and containing every element of $T$.
\State $m\gets\max P$, with $m=0$ if $P$ is empty
\If{$m=0$}
  \State \Return the elements of $P\cup T$ in reverse $\bits$-order
\EndIf
\Statex \textit{Normalize the even and odd subsequences of $P$.}
\State $P_0\gets$ the even entries of $P$, in their original order, divided by $2$
\State $P_1\gets$ the odd entries of $P$, in their original order, mapped by $2x+1\mapsto x$
\Statex \textit{Extend the even part to cover the even targets.}
\State $T_0\gets\{x\in\NN_0:2x\in T\}$
\State $R_0\gets\Call{Extend}{P_0,T_0}$
\State $E_0\gets$ the suffix of $R_0$ obtained by deleting its prefix $P_0$
\State $E\gets E_0$ rescaled by $x\mapsto2x$
\State $h\gets\max E$, with $h=0$ if $E$ is empty
\Statex \textit{Extend the odd part, forcing enough small odd numbers to precede $E$.}
\State $T_1\gets\{x\in\NN_0:2x+1\in T\}\cup\{x\in\NN_0:x<h\}$
\State $R_1\gets\Call{Extend}{P_1,T_1}$
\State $O_1\gets$ the suffix of $R_1$ obtained by deleting its prefix $P_1$
\State $O\gets O_1$ rescaled by $x\mapsto2x+1$
\Statex \textit{Splice the new odd entries before the new even entries.}
\State \Return $P\,O\,E$
\end{algorithmic}
\hrule
\caption[Safe extension procedure $\textsc{Extend}(P,T)$]{%
The recursive construction used in the proof. The two recursive calls are on
words whose largest entries are strictly smaller than $m$. The extra targets
$x<h$ in the odd call force every odd integer at most $2h$ which is not already
in $P$ to be placed before the new even block $E$; doing this prevents
mixed-parity 4APs from forming.}\label{alg:extend}
\end{algorithm}

First apply the induction hypothesis to $P_0$, asking that it contain all
numbers $t/2$ with $t\in T$ even. After rescaling by $x\mapsto2x$, the result
consists of the old even subsequence of $P$, followed by some new even entries;
let $E$ denote the suffix consisting of these new entries. Let $h$ be the
largest entry of $E$, or set $h=0$ if $E$ is empty.

We now deal with the odd entries. Apply the induction hypothesis to $P_1$,
asking that it contain both the normalized odd elements $(t-1)/2$ of $T$ and all
nonnegative integers less than $h$. After rescaling by $x\mapsto2x+1$, the
result consists of the old odd subsequence of $P$, followed by a word $O$ of new
odd entries. In particular, every odd integer at most $2h$ either already occurs
in $P$ or occurs in $O$.

We claim that
$$
Q=P\,O\,E
$$
is the desired extension. By construction, the entries of $O$ and $E$ are new,
and they have opposite parity, so $Q$ has no repeated entries. It begins with
$P$ and contains every element of~$T$. It remains only to prove that $Q$ is
safe.

By construction, after restricting $\C Q$ to the even integers and rescaling
by $2x\mapsto x$, we obtain the safe extension of $P_0$ constructed above,
followed by all remaining nonnegative integers in $\bits$-order. Since that
extension of $P_0$ is safe, this order is 4AP-free. Likewise, after restricting
$\C Q$ to the odd integers and rescaling by $2x+1\mapsto x$, we obtain the safe
extension of $P_1$ constructed above, followed by all remaining nonnegative
integers in $\bits$-order, and hence this order is also 4AP-free. Thus both
parity restrictions of $\C Q$ are 4AP-free. Furthermore, after the old prefix
$P$, the order $\C Q$ has the form
$$
O,\quad E,\quad \text{unused odd integers},\quad
\text{unused even integers},
$$
because all odd integers precede all even integers under $\bits$.

It follows that
\begin{equation}\label{eq:odd-even}
\begin{matrix}
\textsl{if $x$ and $y$ do not occur in $P$, $x$ is even, $y$ is odd,
and $y\le2x$,}\\\textsl{then $y$ precedes $x$ in $\C Q$.}
\end{matrix}
\end{equation}
Indeed, if $x\in E$, then $x\le h$, so every odd $y\le2x$ which does not already
occur in $P$ belongs to $O$. If $x$ is still unused, then every odd integer
occurs before it.

Suppose for contradiction that a 4AP
$$
a,b,c,d
$$
occurs in this order in $\C Q$. If its common difference is even, then all four
terms have the same parity, contradicting the 4AP-freeness of the corresponding
parity restriction. Hence the common difference is odd, so the parities of
$a,b,c,d$ alternate.

First observe that $c$ does not occur in $P$. Otherwise $a,b,c$ all occur in
$P$, since $P$ is an initial segment of $\C Q$. The term $d$ would also follow
$c$ in the old completion $\C P$, whether or not $d$ occurs in $P$. Thus
$a,b,c,d$ would already form a 4AP in $\C P$, contradicting the safety of $P$.
Hence neither $c$ nor $d$ occurs in $P$.

If $c$ is even, then $d$ is odd. Since $a,b,c,d$ is an arithmetic progression,
$$
d=2c-b\le2c.
$$
By~\eqref{eq:odd-even}, $d$ must precede $c$, a contradiction.

It remains to consider the case where $c$ is odd, so that $b$ is even. We claim
that $b$ does not occur in $P$. Suppose otherwise. Since $P$ is an initial
segment of $\C Q$ and $a$ precedes $b$, we then also have $a\in P$. Since we
already know that $c,d\notin P$, in the old completion $\C P$ the terms $a,b$
occur in the initial block $P$, while $c,d$ occur in the $\bits$-ordered tail.
Since $c$ is odd and $d$ is even, $c$ precedes $d$ in that tail. Hence $a,b,c,d$
occur in this order in $\C P$, contradicting the safety of $P$.

Thus $b,c$ do not occur in $P$. But $c=2b-a\le2b$, so~\eqref{eq:odd-even} says
that the odd number $c$ precedes the even number $b$ in $\C Q$, contradicting
the assumed order of the progression in which $b$ precedes $c$.
Therefore $\C Q$ is 4AP-free, and $Q$ is safe. This completes the induction.
\end{proof}

We can now construct the required permutation by successively forcing every
nonnegative integer into the growing safe prefix.

\begin{proof}[Proof of the theorem]
Let $P^{(0)}$ be the empty word. For each $n=0,1,2,\ldots$, apply
Lemma~\ref{lem:extend} to obtain a safe word $P^{(n+1)}$ which has $P^{(n)}$ as
an initial segment and contains $n$.

Thus each $P^{(n)}$ is an initial segment of $P^{(n+1)}$. No entry, once placed,
is ever moved, and every nonnegative integer eventually appears. The union of
these words is therefore a permutation of $\NN_0$.

If this permutation contained a 4AP, its four terms would all occur in some
finite stage~$P^{(n)}$. But every $P^{(n)}$ is safe, and hence in particular
contains no 4AP. This is a contradiction. Finally, adding $1$ to every entry
gives a 4AP-free permutation of the positive integers.
\end{proof}

\medskip
\noindent
The proof of Lemma~\ref{lem:extend} gives a deterministic recursive procedure:
at each recursive call, use exactly the target sets specified in the proof, and
in the base case list exactly the entries of $P\cup T$ in reverse $\bits$-order.
The recursion terminates because the largest entry of the input prefix strictly
decreases in each recursive call, regardless of the size of the target set.

Thus the construction is completely explicit and computable. For instance, if at
stage $n$ we apply this procedure with target set $T=\{n\}$, starting from the
empty word, we obtain a particular 4AP-free permutation of $\NN_0$ beginning
$$
0,1,3,2,11,7,15,9,5,6,4,2063,1039,527,2575,\ldots.
$$
Thus the corresponding permutation of the positive integers begins
$$
1,2,4,3,12,8,16,10,6,7,5,2064,1040,528,2576,\ldots.
$$
Given $j\ge0$, the $j$th term of the permutation can be computed by carrying out
the stages $T=\{0\},\{1\},\{2\},\ldots$ until the current finite prefix has
length greater than $j$; the $j$th entry of that prefix is then the $j$th entry
of the final permutation.

\section*{Declaration of AI usage}
GPT-6 Pro was used to find the construction and to produce an initial draft
of this paper. A formalization of the result is available from
\url{https://github.com/boonsuan/4ap}.

\end{document}